\documentclass[preprint]{elsarticle}
\usepackage{amsmath,amsfonts,amsthm,url,color,amssymb}
\usepackage{slashbox}
\newtheorem{theorem}{Theorem}[section]

\newtheorem{proposition}[theorem]{Proposition}
\newtheorem{lemma}[theorem]{Lemma}
\newtheorem{example}[theorem]{Example}
\newtheorem{define}[theorem]{Definition}
\newtheorem{corollary}[theorem]{Corollary}
\newcommand{\rmv}[1]{}
\newcommand{\I}{\mathcal I}
\newcommand{\Tr}{\mathrm{Tr}}
\newcommand{\fqn}{\mathbb{F}_{q^n}}
\newcommand{\F}{\mathbb{F}}

\title{On $k$-normal elements over finite fields}
\author{Lucas Reis\fnref{fn1}}
\ead{lucasreismat@gmail.com}

\fntext[fn1]{Permanent address: 
Departamento de Matem\'{a}tica,
Universidade Federal de Minas Gerais,
UFMG,
Belo Horizonte MG (Brazil),
 30123-970.}

\address{School of Mathematics and Statistics, Carleton University, 1125 Colonel By Drive, Ottawa ON (Canada), K1S 5B6}

\begin{document}

\begin{abstract}
The so called $k$-normal elements appear in the literature as a generalization of normal elements over finite fields. Recently, questions concerning the construction of $k$-normal elements and the existence of $k$-normal elements that are also primitive have attracted attention from many authors. In this paper we give alternative constructions of $k$-normal elements and, in particular, we obtain a sieve inequality for the existence of primitive, $k$-normal elements. As an application, we show the existence of primitive $k$-normal elements for a significant proportion of $k$'s in many field extensions. In particular, we prove that there exist primitive $k$-normals in $\F_{q^n}$ over $\F_q$ in the case when $k$ lies in the interval $[1, n/4]$, $n$ has a special property and $q, n\ge 420$.
\end{abstract}


\maketitle

\section{Introduction}
Let $\F_{q^n}$ be the finite field with $q^n$ elements, where $q$ is a prime power and $n$ is a positive integer. We have two special notions of generators in the theory of finite fields. The multiplicative group $\F_{q^n}^*$ is cyclic, with $q^n-1$ elements, and any generator is called {\it primitive}. Also, $\F_{q^{n}}$ can be regarded as an $\F_q$-vector space over $\F_q$: its dimension is $n$ and, in particular, $\F_{q^n}$ is isomorphic to $\F_q^n$. An element $\alpha\in \F_{q^n}$ is said to be {\it normal} over $\F_q$ if $A=\{\alpha, \alpha^q, \cdots, \alpha^{q^{n-1}}\}$ is a basis of $\F_{q^n}$ over $\F_q$: A is frequently called a {\it normal basis}. Due to their high efficiency, normal bases are frequently used in cryptography and computer algebra systems; sometimes it is also interesting to use normal bases composed by primitive elements. The {\it Primitive Normal Basis Theorem} states that for any extension field $\F_{q^n}$ of $\F_q$, there exists a normal basis composed by primitive elements; this result was first proved by Lenstra and Schoof \cite{lenstra} and a proof without the use of a computer was later given in \cite{cohen}. 

Recently, Huczynska {\it et al} \cite{HMPT} introduce $k$-normal elements, extending the notion of normal elements. There are many equivalent definitions and here we present the most natural in the sense of vector spaces.

\begin{define}
For $\alpha\in \F_{q^n}$, consider the set $S_{\alpha}=\{\alpha, \alpha^q, \cdots, \alpha^{q^{n-1}}\}$ comprising the conjugates of $\alpha$ by the action of the Galois Group of $\F_{q^n}$ over $\F_q$. The element $\alpha$ is said to be $k$-normal over $\F_q$ if the vector space $V_{\alpha}$ generated by $S_{\alpha}$ has dimension $n-k$, i.e., $V_{\alpha}\subseteq \F_{q^n}$ has co-dimension $k$. 
\end{define}

From definition, $0$-normal elements correspond to normal elements in the usual sense. Also, the concept of $k$-normal depends strongly on the base field that we are working. For this reason, unless otherwise stated, $\alpha\in \F_{q^n}$ is $k$-normal if it is $k$-normal over $\F_q$. 

In \cite{HMPT}, the authors find a formula for the number of $k$-normals and, consequently, obtain some results on the density of these elements. Motivated by the Primitive Normal Basis Theorem, they obtain an existence result on primitive, $1$-normal elements.

\begin{theorem}[\cite{HMPT}, Theorem 5.10] \label{mainpanario}
Let $q=p^e$ be a prime power and $n$ a positive integer not divisible by $p$. Assume that $n\ge 6$ if $q\ge 11$ and that $n\ge 3$ if $3\le q\le 9$. Then there exists a primitive $1$-normal element of $\F_{q^n}$ over $\F_q$.
\end{theorem}

In \cite{R17}, we explore some ideas of \cite{HMPT} and, in particular, we obtain a characterization of $k$-normals in the case when $n$ is not divisible by $p$. In the same paper, we find some asymptotic results on the existence of $k$-normal elements that are primitive or have a reasonable high multiplicative order. In \cite{Alizadeh}, the author obtains alternative characterizations of $k$-normals via some recursive constructions and, in particular, he presents a method for constructing $1$-normal elements in even characteristic. Recently, Theorem \ref{mainpanario} was extended to arbitrary $n\ge 3$ (see \cite{RT}), where the authors use many techniques for completing the case $\gcd(n, p)=1$ and extending this existence result to the case when $n$ is divisible by $p$. 

In this paper we deal with the general question concerning the existence of primitive, $k$-normal elements. We introduce a particular class of $k$-normal elements, obtained from normal elements, that removes the pertinent obstruction $n\equiv 0\pmod p$. In particular, we obtain a character sum formula, derived from the Lenstra-Schoof method, for the number of certain primitive $k$-normals. As an application, we obtain some existence results on primitive $k$-normals in special extensions of $\F_q$. We also extend some results of \cite{RT} to $k$-normal elements.

\section{Preliminaries}
Here we introduce some basic definitions related to the $k$-normal elements, adding some recent results. We start with some arithmetic functions and their polynomial version.

\begin{define}
\begin{enumerate}[(a)]
\item Let $f(x)$ be a monic polynomial with coefficients in $\F_q$. The Euler Phi Function for polynomials over $\F_q$ is given by $$\Phi(f)=\left |\left(\frac{\F_q[x]}{\langle f\rangle}\right)^{*}\right |,$$ where $\langle f\rangle$ is the ideal generated by $f(x)$ in $\F_q[x]$. 
\item If $t$ is a positive integer (or a monic polynomial over $\F_q$), $W(t)$ denotes the number of square-free (monic) divisors of $t$.
\item If $f(x)$ is a monic polynomial with coefficients in $\F_q$, the Polynomial Mobius Function $\mu_q$ is given by $\mu_q(f)=0$ is $f$ is not square-free and $\mu_q(f)=(-1)^r$ if $f$ writes as a product of $r$ distinct irreducible factors over $\F_q$.  
\end{enumerate}
\end{define}

\subsection{$q$-polynomials and $k$-normals}
For $f\in \F_q[x]$, $f=\sum_{i=0}^{s}a_ix^i$, we set $L_f(x)=\sum_{i=0}^{s}a_ix^{q^i}$ as the $q$-associate of $f$. Also, for $\alpha\in \F_{q^n}$, we set $f\circ \alpha=L_{f}(\alpha)=\sum_{i=0}^sa_i\alpha^{q^i}$.

As follows, we have some basic properties of the $q$-associates:

\begin{lemma}(\cite{LN}, Theorem 3.62)
Let $f, g\in \F_q[x]$. The following hold:
\begin{enumerate}[(i)]
\item $L_f(L_g(x))=L_{fg}(x)$,
\item $L_f(x)+L_g(x)=L_{f+g}(x)$.
\end{enumerate}
\end{lemma}

For $\alpha\in \F_{q^n}$, we define $\I_{\alpha}$ as the subset of $\F_q[x]$ comprising the polynomials $f(x)$ for which $f(x)\circ \alpha=0$, i.e., $L_{f}(\alpha)=0$. Notice that $x^{n}-1\in \I_{\alpha}$ and, from the previous Lemma, it can be verified that $\I_{\alpha}$ is an ideal of $\F_q[x]$, hence $\I_{\alpha}$ is generated by a polynomial $m_{\alpha}(x)$. We can suppose $m_{\alpha}(x)$ monic. The polynomial $m_{\alpha}(x)$ is defined as the $\F_q$-order of $\alpha$. This is a dual definition of multiplicative order in $\F_{q^n}^*$. 

Clearly $m_{\alpha}(x)$ is always a divisor of $x^n-1$, hence its degree is $j$ for some $0\le j\le n$. Notice that $j=0$ if and only if $m_{\alpha}(x)=1$, i.e., $\alpha=0$. The following result shows a connection between $k$-normal elements and their $\F_q$-order.

\begin{proposition}(\cite{HMPT}, Theorem 3.2)
Let $\alpha \in \F_{q^n}$. Then $\alpha$ is $k$-normal if and only if $m_{\alpha}(x)$ has degree $n-k$. 
\end{proposition}

In particular, an element $\alpha $ is normal if and only if $m_{\alpha}(x)=x^n-1$. We see that the existence of $k$-normals depends on the existence of a polynomial of degree $n-k$ dividing $x^n-1$ over $\F_q$; as follows, we have a formula for the number of $k$-normal elements.

\begin{lemma}(\cite{HMPT}, Theorem 3.5) \label{count} The number $N_k$ of $k-$normal elements of $\F_{q^n}$ over $\F_q$ is given by
\begin{equation}N_k=\label{eq2}\sum_{h|x^n-1\atop{\deg(h)=n-k}}\Phi(h),\end{equation}
where the divisors are monic and polynomial division is over $\F_q$.
\end{lemma}
One of the main steps in the proof of the result above is the fact that, for each monic divisor $f(x)$ of $x^n-1$, there exists $\Phi_q(f(x))$ elements $\alpha\in \F_{q^n}$ for which $m_{\alpha}(x)=f(x)$. This fact will be further used in this paper.

Since $\frac{x^n-1}{x-1}$ divides $x^n-1$, there are $1$-normal elements over any finite field extension. However, the sum in equality \eqref{eq2} can be empty: for instance, if $q=5$ and $n=7$: $$x^7-1=(x-1)(x^6+x^5+x^4+x^3+x^2+x+1)$$ is the factorization of $x^7-1$ into irreducible factors over $\F_5$. In particular, there are no $2$, $3$, $4$ or $5$-normal elements of $\F_{5^7}$ over $\F_{5}$. More generally, if $n$ is a prime and $q$ is primitive~$\mod n$, $x^n-1$ factors as $(x-1)(x^{n-1}+\cdots+x+1)$ and we do not have $k$-normal elements for any $1<k<n-1$.  The existence of $k$-normals is not guaranteed for generic values of $k$.

As follows, we may construct $k$-normal elements from a given normal element.

\begin{lemma}\label{LR1}
Let $\beta\in \F_{q^n}$ be a normal element over $\F_q$ and $f(x)$ be a polynomial of degree $k$ such that $f(x)$ divides $x^n-1$. Then $\alpha=f(x)\circ \beta$ is $k$-normal.
\end{lemma}

\begin{proof}
We prove that $m_{\alpha}(x)=\frac{x^n-1}{f(x)}$ and this implies the desired result. Notice that $\frac{x^n-1}{f(x)}\circ \alpha=\frac{x^n-1}{f(x)}\circ (f(x)\circ \beta)=(x^n-1)\circ \beta=0$, hence $m_{\alpha}(x)$ divides $\frac{x^n-1}{f(x)}$. It cannot divide strictly because $m_{\beta}(x)=x^n-1$ and this completes the proof. 
\end{proof}

In particular, we have a method for constructing $k$-normal elements when they exist: if we find a divisor $f(x)$ of $x^n-1$ of degree $k$ and a normal element $\beta\in \F_{q^n}$, the element $\alpha=f(x)\circ \beta$ is $k$-normal. There are many ways of finding normal elements in finite field extensions,  including constructive and random methods; this is a classical topic in the theory of finite fields and the reader can easily find a wide variety of papers regarding those methods. For instance, see \cite{zurG}.

\subsection{A characteristic equation for elements with prescribed $\F_q$-order}
We have noticed that we may construct $k$-normal elements from the normal elements. However, it is not guaranteed that this method describes every $k$-normal in $\F_{q^n}$. For a polynomial $f$ dividing $x^n-1$, set $$\Psi_f(x)=\prod_{m_{\alpha}=f(x)}(x-\alpha),$$ the polynomial of least degree that vanishes in every element $\alpha\in \F_{q^n}$ with $m_{\alpha}=f(x)$. Clearly $m_{\alpha}=f$ if and only if $\Psi_f(\alpha)=0$. Also, for $f(x)\in \F_{q}[x]$ and $\alpha\in \F_{q^n}$, we have $f(x)\circ \alpha=L_f(\alpha)=0$ if and only if $m_{\alpha}$ divides $f(x)$. In particular, this shows that $L_{f}(x)=\prod_{g|f}\Psi_g(x)$. This identity is similar to the one describing cyclotomic polynomials $$x^n-1=\prod_{n|d}\Phi_d(x).$$ From the Mobius Inversion formula, we may deduce $\Phi_d(x)=\prod_{r|d}(x^r-1)^{\mu(d/r)}$. This last equality describes the elements of multiplicative order $d$ in finite fields. Motivated by this characterization, we obtain the following:

\begin{proposition}
Let $f(x)$ be any divisor of $x^n-1$ over $\F_{q}$. The following holds:

\begin{equation}\label{charF_q}\Psi_f(x)=\prod_{g|f}L_g(x)^{\mu_q(f/g)},\end{equation}
where $g$ is monic and polynomial division is over $\F_q$. 
\end{proposition}

\begin{proof}
Notice that 
$$\prod_{g|f}L_g(x)^{\mu_q(f/g)}=\prod_{g|f}L_{f/g}(x)^{\mu_q(g)},$$
and, from $L_{f/g}(x)=\prod_{h|f/g}\Psi_h(x)$, we obtain
$$\prod_{g|f}L_{f/g}(x)^{\mu_q(g)}=\prod_{h|f}\Psi_{h}(x)^{\sum_{g|f/h}\mu_q(g)}.$$
Writing $f/h$ as product of irreducibles over $\F_q$, we can easily see that $$\sum_{g|f/h}\mu_q(g)=\begin{cases} 1 & \text{if  $f/h=1$,} \\ 0 & \text{otherwise.}\end{cases}.$$ This shows that $$\prod_{g|f}L_g(x)^{\mu_q(f/g)}=\Psi_f(x).$$
\end{proof}

In particular, if we set 
$$\Lambda_k(x)=\prod_{f|x^n-1\atop \deg f=n-k}\Psi_f(x),$$ an element $\alpha\in \F_{q^n}$ is $k$-normal if and only if $\Lambda_k(\alpha)=0$. 
For instance, we may characterize the normal elements in $\F_{q^n}$: an element $\alpha\in \F_{q^n}$ is normal if and only if $\Lambda_0(\alpha)=0$, where
$$\Lambda_{0}(x)=\Psi_{x^n-1}(x)=\prod_{g|x^n-1}L_g(x)^{\mu_q(x^n-1/g)}.$$

\begin{example}
Suppose that $q\equiv 3\pmod 4$ and $n=4$. Then $\Lambda_0(x)=\frac{x^{q^4-1}-1}{(x^{q^2-1}+1)(x^{q^2-1}-1)}=\sum_{i=0}^{\frac{q^2-1}{2}}x^{2i(q^2-1)}$. Hence, $\alpha\in \F_{q^4}$ is normal if and only if 
$$\sum_{i=0}^{\frac{q^2-1}{2}}\alpha^{2i(q^2-1)}=0.$$
\end{example}

In general, we have shown that the $k$-normals can be described as the zeroes of a univariate polynomial over $\F_q$; this polynomial can be computed from the factorization of $x^n-1$ over $\F_q$.

\subsection{Some recent results}
The proof of Theorem \ref{mainpanario} is based in an application of the {\it Lenstra-Schoof} method, introduced in \cite{lenstra}; this method has been used frequently in the characterization of elements in finite fields with particular properties like being primitive, normal and of zero-trace. In particular, from Corollary 5.8 of \cite{HMPT}, we can easily deduce the following:

\begin{lemma}\label{aux2}
Suppose that $q$ is a power of a prime $p$, $n\ge 2$ is a positive integer not divisible by $p$ and $T(x)=\frac{x^n-1}{x-1}$. If
\begin{equation}\label{sieve}W(T)\cdot W(q^n-1)< q^{n/2-1}, \end{equation}
there exist primitive, $1$-normal elements of $\F_{q^n}$ over $\F_q$.
\end{lemma}

Inequality \eqref{sieve} is an essential step in the proof of Theorem \ref{mainpanario} and it was studied in \cite{cohen2}; if $n\ge 6$ for $q\ge 11$ and $n\ge 3$ for $3\le q\le 9$, this inequality holds for all but a finite number of pairs $(q, n)$.

In \cite{HMPT}, the authors propose an extension of the Theorem \ref{mainpanario} for all pairs $(q, n)$ with $n\ge 3$ as a problem (\cite{HMPT}, Problem 6.2); they conjectured that such elements always exist. This was recently proved in \cite{RT}, where the authors use many different techniques  to complete Theorem \ref{mainpanario} in the case $\gcd(n, p)=1$ and add the case when $p$ divides $n$. 

In particular, when $n$ is divisible by $p^2$, they obtain the following:

\begin{lemma}(\cite{RT}, Lemma 5.2)\label{projection-}
Suppose that $\F_q$ has characteristic $p$ and let $n=p^2s$ for any $s \geq 1$. Then $\alpha \in \F_{q^n}$ is such that $m_{\alpha}(x) = \frac{x^n-1}{x-1}$ if and only if $\beta = \Tr_{q^n/q^{ps}}(\alpha)=\sum_{i=0}^p\alpha^{q^{psi}}$ satisfies $m_{\beta}(x) = \frac{x^{ps}-1}{x-1}$. 
\end{lemma}

Using the well-known result on the existence of primitive elements with prescribed trace due to Cohen (see \cite{Cohentrace}), they prove the existence of primitive $1$-normals in the case when $p^2$ divides $n$. The case $n=ps$ with $\gcd(p, s)=1$ is dealt in a similar way of \cite{HMPT}.

As follows, we have a generalization of Lemma~\ref{projection-}.

\begin{lemma}(\cite{RT}, Lemma 5.2)\label{projection}
Suppose that $\F_q$ has characteristic $p$ and let $n=p^2s$ for any $s \geq 1$. Also, let $f(x)$ be a polynomial dividing $x^{s}-1$. Then $\alpha \in \F_{q^n}$ is such that $m_{\alpha}(x) = \frac{x^n-1}{f(x)}$ if and only if $\beta = \Tr_{q^n/q^{ps}}(\alpha)=\sum_{i=0}^{p-1}\alpha^{q^{psi}}$ satisfies $m_{\beta}(x) = \frac{x^{ps}-1}{f(x)}$. 
\end{lemma}

The proof is entirely similar to the proof of Lemma~\ref{projection-} (see \cite{RT}) so we will omit. Motivated by Theorem 5.3 of \cite{RT}, we have the following.

\begin{theorem}\label{p2prim}
Suppose that $n=p^2\cdot s$ and $x^s-1$ is divisible by a polynomial of degree $k$. Then there exists a primitive, $k$-normal element over $\F_q$.
\end{theorem}

\begin{proof}
Let $f(x)$ be a polynomial of degree $k$ such that $f(x)$ divides $x^s-1$. In particular, $f(x)$ divides $x^{ps}-1$, i.e., $g(x)=\frac{x^{ps}-1}{f(x)}$ is a polynomial. Since $g(x)$ divides $x^{ps}-1$, we know that there exists an element $\beta\in \F_{q^{ps}}$ such that the $\F_q$-order $m_{\beta}(x)$ of $\beta$ is $g(x)$. Clearly $g(x)\ne 1$, hence $\beta \ne 0$. According to \cite{Cohentrace}, there exists a primitive element $\alpha\in \F_{q^n}$ such that $\Tr_{q^n/q^{ps}}(\alpha)=\beta$ and then, from Lemma~\ref{projection}, it follows that such an $\alpha$ satisfies $m_{\alpha}=\frac{x^n-1}{f(x)}$, i.e., $\alpha$ is a primitive, $k$-normal element. 
\end{proof}

\section{Characteristic function for a class of primitive $k$-normals}
In this section, we use the method of Lenstra and Schoof in the characterization of primitive and normal elements. This method has been used by many different authors in a wide variety of existence problems. For this reason, we skip some details, which can be found in \cite{cohen}. We recall the notion of {\it freeness}.

\begin{define}
\begin{enumerate}
\item If $m$ divides $q^n-1$, an element $\alpha \in \F_{q^n}^*$ is said to be $m$-free if $\alpha = \beta^d$ for any divisor $d$ of $m$ implies $d=1$. 
\item If $m(x)$ divides $x^n-1$, an element $\alpha\in \F_{q^n}$ is $m(x)$-free if $\alpha = h \circ \beta$ for any divisor $h(x)$ of $m(x)$ implies $h=1$. 
\end{enumerate}
\end{define}
It follows from definition that primitive elements correspond to the $(q^n-1)$-free elements. Also, $\alpha\in \F_{q^n}$ is normal if and only if is $(x^n-1)$-free. The concept of {\it freeness} derives some characteristic functions for primitive and normal elements. We pick the notation of \cite{HMPT}.
\vspace{.2 cm}

\noindent\textbf{Multiplicative Part}: $\int\limits_{d|m}\eta_d$ denotes $\sum_{d|m}\frac{\mu(d)}{\varphi(d)}\sum_{(d)}\eta_d$, where $\mu$ and $\varphi$ are the Mobius and Euler functions for integers, respectively, $\eta_d$ is a typical multiplicative character of $\F_{q^n}$ of order $d$, and the sum $\sum_{(d)}\eta_d$ runs through all the multiplicative characters of order $d$.
\vspace{.2 cm}

\noindent\textbf{Additive Part}: $\chi$ denotes the canonical additive character on $\F_{q^n}$, i.e. $$\chi(\omega)=\lambda\left(\sum_{i=0}^{n-1}\omega^{q^i}\right), \omega\in \F_{q^n},$$ where $\lambda$ is the canonical additive character of $\F_q$ to $\F_{p}$. If $D$ is a monic polynomial dividing $x^n-1$ over $\F_q$, a typical character $\chi_*$ of $\F_{q^n}$ of $\F_q$-order $D$ is one such that $\chi_*(D\circ^{q} \cdot)$ is the trivial additive character in $\F_{q^n}$ and $D$ is minimal (in terms of degree) with this property. Let $\Delta_D$ be the set of all $\delta\in \F_{q^n}$ such that $\chi_{\delta}$ has $\F_q$-order $D$, where $\chi_{\delta}(\omega)=\chi(\delta\omega)$ for any $\omega\in \F_{q^n}$. For instance, $\Delta_1=\{0\}$ and $\Delta_{x-1}=\F_q^*$.

In the same way of the multiplicative part, $\int\limits_{D|T}\chi_{\delta_D}$ denotes the sum $$\sum_{D|T}\frac{\mu_q(D)}{\Phi(D)}\sum_{(\delta_D)}\chi_{\delta_D},$$ where $\mu_q$ and $\Phi$ are the Mobius and Euler functions for polynomials over $\F_q$, respectively, $\chi_{\delta_D}$ denotes a typical additive character of $\F_{q^n}$ of $\F_q$-Order $D$ and the sum $\sum_{(\delta_D)}\chi_{\delta_D}$ runs through all the additive characters whose $\F_q$-order equals $D$, i.e., $\delta_D\in \Delta_D$.

For $t$ dividing $q^n-1$ and $D$ dividing $x^n-1$, set $\theta(t)=\frac{\varphi(t)}{t}$ and $\Theta(D)=\frac{\Phi(D)}{q^{\deg D}}$. 

\begin{theorem}\textup{\cite[Section 5.2]{HMPT}}\label{thm:charfree}
\begin{enumerate}
\item For $w \in \fqn^*$ and $t$ be a positive divisor of $q^n-1$, 
\[\omega_t(w) = \theta(t) \int_{d|t} \eta_{d}(w) = \begin{cases} 1 & \text{if $w$ is $t$-free,} \\ 0 & \text{otherwise.} \end{cases}\]
\item For $w \in \fqn$ and $D$ be a monic divisor of $x^n-1$,
\[\Omega_D(w) = \Theta(D) \int_{E|D} \chi_{\delta_E}(w) = \begin{cases} 1 & \text{if $w$ is $D$-free,} \\ 0 & \text{otherwise.} \end{cases}\]
\end{enumerate}
\end{theorem}

In particular, for $t=q^n-1$ and $D=x^n-1$, we obtain characteristic functions for primitive and normal elements, respectively. We write $\omega_{q^n-1}=\omega$ and $\Omega_{x^n-1}=\Omega$. As usual, we may extend the multiplicative characters to $0$ by setting $\eta_1(0)=1$, where $\eta_1$ is the trivial multiplicative character and $\eta(0)=0$ if $\eta$ is not trivial.

We have seen that some $k$-normals arise from a normal element $\beta$ via the composition $f(x)\circ \beta$, where $f(x)$ is a divisor of $x^n-1$, $\deg f=k$; more than that, $\alpha=f(x)\circ \beta$ satisfies $m_{\alpha}=\frac{x^n-1}{f(x)}$. In particular, we may obtain the characteristic function for the primitive elements arising from this construction: note that $\Omega(w)\cdot \omega(L_f(w))=1$ if and only if $w$ is normal and $L_f(w)=f(x)\circ w$ is primitive. This characteristic function describes a particular class of primitive $k$-normal elements. The following is straightforward.

\begin{proposition}\label{charfun}
Let $f(x)$ be a divisor of $x^n-1$ of degree $k$ and $n_f$ be the number of primitive elements of the form $f(x)\circ \alpha$, where $\alpha$ is a normal element. The following holds:
\begin{equation}\label{Char}\frac{n_f}{\theta(q^n-1)\Theta(x^n-1)}=\sum\limits_{w\in \F_{q^n}}\displaystyle\int\limits_{d|q^n-1}\displaystyle\int\limits_{D|x^n-1}\eta_d(L_{f}(w))\chi_{\delta_D}(w).\end{equation}
In particular, the number of primitive, $k$-normal elements in $\F_{q^n}$ is at least $n_f$. 
\end{proposition}

\subsection{Character sums and a sieve inequality}
Here we make some estimates for the character sums that appear naturally from Eq.~\eqref{Char}: note that $\eta_d$ is the trivial multiplicative character if and only if $d=1$. Also, $\chi_{\delta_D}$ is the trivial additive character if and only if $\delta_D=0$, i.e., $D=1$.

As usual, we split the sum in Eq.~\eqref{Char} as Gauss sums types, according to the trivial and non-trivial characters. For each $d$ dividing $q^n-1$ and $D$ dividing $x^n-1$, set $G_f(\eta_d, \chi_{\delta_D})=\sum\limits_{w\in \F_{q^n}}\eta_d(L_{f}(w))\chi_{\delta_D}(w)$.

Note that, from Proposition~\ref{charfun}, $$\frac{n_f}{\theta(q^n-1)\Theta(x^n-1)}=s_0+S_1+S_2+S_3,$$ where $s_0=G_f(\eta_1, \chi_0)$, $S_1=\displaystyle\int\limits_{D|x^n-1\atop D\ne 1}G_f(\eta_1, \chi_{\delta_D})$, $S_2=\displaystyle\int\limits_{d|q^n-1\atop d\ne 1}G_f(\eta_d, \chi_{0})$ and $S_3=\displaystyle\int\limits_{d|q^n-1\atop d\ne 1}\displaystyle\int\limits_{D|x^n-1\atop D\ne 1}G_f(\eta_d, \chi_{\delta_D})$.

From definition, $s_0=q^n$. Also, note that $$G_f(\eta_1, \chi_{\delta_D})=\sum\limits_{w\in \F_{q^n}}\eta_1(L_{f}(w))\chi_{\delta_D}(w)=\sum\limits_{w\in \F_{q^n}}\chi_{\delta_D}(w)=0,$$ for any divisor $D$ of $x^n-1$ with $D\ne 1$. In particular, $S_1=0$. For the quantities $S_2$ and $S_3$, we use some general bounds on character sums.

\begin{lemma}(\cite{LN}, Theorem 5.41)\label{Gauss1} Let $\eta$ be a multiplicative character of $\F_{q^n}$ of order $r>1$ and $f\in \F_{q^n}[x]$ be a monic polynomial of positive degree such that $f$ is not of the form $g(x)^r$ for some $g\in \F_{q^n}[x]$ with degree at least $1$. Suppose that $e$ is the number of distinct roots of $f$ in its splitting field over $\F_{q^n}$. For every $a\in \F_{q^n}$, 
$$\left|\sum_{c\in \F_{q^n}}\eta(af(c))\right|\le (e-1)q^{n/2}.$$
\end{lemma}

\begin{lemma}(\cite{CM00})\label{Gauss2} Let $\eta$ be a non trivial multiplicative character of order $r$ and $\chi$ be a nontrivial additive character of $\F_{q^n}$. Let $f$ and $g$ be rational functions in $\F_{q^n}(x)$ such that $f$ is not of the form $y\cdot h^r$ for any $y\in \F_{q^n}$  and $h\in \F_{q^n}(x)$, and $g$ is not of the form $h^p-h+y$ for any $y\in \F_{q^n}$ and $h\ge \F_{q^n}(x)$. Then, 
$$\left|\sum_{w\in \F_{q^n}\setminus S}\eta(f(w))\cdot \chi(g(w))\right|\le (\deg (g)_{\infty}+m+m'-m''-2)q^{n/2},$$

where $S$ is the set of poles of $f$ and $g$, $(g)_{\infty}$ is the pole divisor of $g$, $m$ is the number of distinct zeros and finite poles of $f$ in the algebraic closure $\overline{\F}_{q}$ of $\F_q$, $m'$ is the number of distinct poles of $g$ (including $\infty$) and $m''$ is the number of finite poles of $f$ that are poles or zeros of $g$. 
\end{lemma}

Write $f(x)=\sum_{i=0}^k a_ix^i$. Note that, since $L_f$ is a $q$-polynomial, its formal derivative equals $a_0$. In particular, $L_f$ does not have repeated roots, hence cannot be of the form $y\cdot g(x)^r$ for some $r>1$. 

Also, if $f(x)$ divides $x^n-1$ and has degree $k$, we know that $L_f=0$ has exactly $q^{k}$ roots over $\F_{q^n}$; these roots describe a $k$-dimensional $\F_q$-vector subspace of $\F_{q^n}$. Finally, notice that $g(x)=x$ cannot be written as $h^p-h-y$ for any rational function $h\in \F_{q^n}(x)$ and $y\in \F_{q^n}$. From Lemma~\ref{Gauss1} we conclude that, for $d>1$ a divisor of $q^n-1$, $$|G_f(\eta_d, \chi_0)|=\left|\sum_{c\in \F_{q^n}}\eta_d(L_f(c))\right|\le (q^k-1)q^{n/2}.$$

From Lemma~\ref{Gauss2}, it follows that, for any $D$ divisor of $x^n-1$ and $d$ a divisor of $q^n-1$, with $D, d\ne 1$, 
$$|G_f(\eta_d, \chi_{\delta_D})|=\left|\sum_{w\in \F_{q^n}}\eta_d(L_f(w))\chi_{\delta_D}(w)\right|\le (1+q^k+1-0-2)q^{n/2}=q^{n/2+k}.$$

Combining all the previous bounds, we obtain the following:

\begin{theorem}\label{main}
Let $f(x)$ be a divisor of $x^n-1$ of degree $k$ and let $n_f$ be the number of primitive elements of the form $f(x)\circ \alpha$, where $\alpha$ is a normal element. The following holds:
$$\label{charac}\frac{n_f}{\theta(q^n-1)\Theta(x^n-1)}>q^n-q^{n/2+k}W(q^n-1)W(x^n-1).$$
In particular, if 
\begin{equation}\label{sieveLR}  q^{n/2-k}\ge W(q^n-1)W(x^n-1),\end{equation}
then there exist primitive $k$-normal elements in $\F_{q^n}$.
\end{theorem}

\begin{proof}
We have seen that $\label{charac}\frac{n_f}{\theta(q^n-1)\Theta(x^n-1)}=s_0+S_1+S_2+S_3=q^n+S_2+S_3$. In particular, 

$$\frac{n_f}{\theta(q^n-1)\Theta(x^n-1)}\ge q^n-|S_2|-|S_3|.$$

Applying estimates to the sums $S_2, S_3$, we obtain $$|S_2|\le (W(q^n-1)-1)(q^k-1)q^{n/2}$$ and $|S_3|\le (W(q^n-1)-1)(W(x^n-1)-1)q^{n/2+k}$, hence $$|S_1|+|S_2|<q^{n/2+k}W(q^n-1)W(x^n-1).$$ In particular, if Eq.~\eqref{sieveLR} holds, $n_f>0$ and we know that the number of primitive $k$-normals is at least $n_f$. This completes the proof.

\end{proof}

Notice that Eq.~\eqref{sieveLR} generalizes the sieve inequality in Lemma~\ref{aux2}. 

\section{Existence of primitive $k$-normals}
In this section, we discuss the existence of primitive $k$-normals in some special extensions of $\F_q$. Essentially, we explore the sieve inequality present in Theorem~\ref{main} and the result contained in Theorem~\ref{p2prim}. 

We have seen that the existence of $k$-normals is not always ensured for generic values of $1<k<n-1$: from Lemma~\ref{count}, the number of $k$-normals is strongly related to the factorization of $x^n-1$ over $\F_q$. This motivates us to introduce the concept of practical numbers.

\begin{define}(Practical numbers)
A positive integer $n$ is said to be $\F_q$-practical if, for any $1\le k\le n-1$, $x^n-1$ is divisible by a polynomial of degree $k$ over $\F_q$.
\end{define}
Notice that, if $n$ is $\F_q$-practical, there exist $k$-normals in $\F_{q^n}$ for any $1<k<n-1$.  This definition arises from the so called $\varphi$-{\it practical numbers}: they are the positive integers $n$ for which $x^n-1\in \mathbb Z[x]$ is divisible by a polynomial of degree $k$ for any $1\le k\le n-1$. These $\varphi$-practical numbers have been extensively studied in many aspects, such as their density over $\mathbb N$ and their asymptotic number. In particular, if $s(t)$ denotes the number of $\varphi$-practical numbers up to $t$, according to \cite{practical}, there exists a constant $C>0$ such that $\lim\limits_{t\to\infty}\frac{s(t)\cdot \log t}{t}=C$. This shows that the $\varphi$-practical numbers behaves like the primes on integers and, in particular, their density in $\mathbb N$ is zero.

Notice that the factorization of $x^n-1$ over $\mathbb Z$ also holds over any finite field: we take the coefficients $\pmod p$ and recall that $\F_p\subseteq \F_q$. This shows that any $\varphi$-practical number is also $\F_q$-practical. In particular, the number of $\F_q$-practicals up to $t$ has growth at least $\frac{Ct}{\log t}$. The exact growth of the number of $\F_q$-practicals is still an open problem. However, we can find infinite families of such numbers.

\begin{proposition}(\cite{R17}, Theorem 4.4)\label{practicals}
Let $q$ be a power of a prime $p$ and let $n$ be a positive integer such that every prime divisor of $n$ divides $p(q-1)$. Then $n$ is $\F_q$-practical.
\end{proposition}

From Theorem~\ref{p2prim}, we obtain the following.

\begin{corollary}\label{corpractical}
Let $n=p^2s$, where $s$ is an $\F_q$-practical number. Then, for any $1\le k\le s$, there exists a primitive, $k$-normal element of $\F_{q^n}$. 
\end{corollary}

The previous corollary ensures the existence of primitive $k$-normals for $k$ in the interval $[1, \frac{n}{p^2}]$. This corresponds to a proportion close to $1/p^2$ of the possible values of $k$. 

In the rest of this paper, we explore Eq.~\eqref{sieveLR} in extensions of degree $n$, where $n$ is an $\F_q$-practical number. Essentially, we obtain effective bounds on the functions $W(x^n-1)$ and $W(q^n-1)$. 
\subsection{Some estimates for the square-free divisors counting}
Here we obtain some estimates on the functions $W(q^n-1)$ and $W(x^n-1)$. We start with some general bounds. 
The bound $W(x^n-1)\le 2^{n}=q^{n\log_q 2}$ is trivial. Also, we have a general bound for $W(q^n-1)$: if $d(q^n-1)$ denotes the number of divisors of $q^n-1$, $W(q^n-1)\le d(q^n-1)$. For the number of divisors function, we have the following.

\begin{lemma}\label{divisor} If $d(m)$ denotes the number of divisors of $m$, then for all $m\ge 3$,
$$d(m)\le m^{\frac{1.5379 \log 2}{\log\log m}}<m^{\frac{1.06}{\log\log m}}.$$
\end{lemma}
\begin{proof}This inequality is a direct consequence of the result in \cite{divisor}.
\end{proof}

In particular, we can easily obtain $W(q^n-1)<q^{\frac{1.06n}{\log\log (q^n-1)}}$. We now study the special case when $n$ is a power of two.

\begin{lemma}
Suppose that $q$ is odd. For any $i\ge 1$, if $r_0$ is an odd prime that divides $q^{2^i}+1$ , then $r_0\equiv 1\pmod{2^{i+1}}$.
\end{lemma}

\begin{proof}
Let $l$ be the least positive integer such that $q^l\equiv1 \pmod {r_0}$. Clearly $l$ divides $\varphi(r_0)=r_0-1$. Also, $r_0$ divides $q^{2^{i+1}}-1$ but not $q^{2^{i}}-1$, it follows that $l$ divides $2^{i+1}$ but not $2^{i}$. This shows that $2^{i+1}=l$, hence $2^{i+1}$ divides $r_0-1$.
\end{proof}

In particular, we obtain the following.
\begin{proposition}\label{bound1}
For any $q\ge 3$ odd and $t\ge 2$,
\begin{equation}\label{bound}W(q^{2^t}-1)< 2q^{\frac{2^t}{t-1}}.\end{equation}
\end{proposition}

\begin{proof}
For $1\le i<t-1$, set $d_i=q^{2^i}+1$. Note that $W(q^{2^t}-1)\le W(q^2-1)\cdot \prod_{i=2}^{t-1}W(d_i/2)$. For a fixed $2\le i\le t$, let $s_1< \cdots< s_{d(i)}$ be the distinct odd primes that divide $d_i$. Clearly $W(d_i/2)=2^{d(i)}$. As we have seen, $s_j\equiv 1\pmod {2^{i+1}}$, hence $s_j>2^{i+1}$ and then
$$d_i> s_1\cdots s_{d(i)}> 2^{(i+1)\cdot d(i)},$$
hence $2^{d(i)}<q^{\frac{2^i}{i+1}}$.
It follows by induction that $\sum_{i=1}^{t-1}\frac{2^i}{i+1}\le \frac{2^t}{t-1}-1$ for $t\ge 2$. Moreover, we have the trivial bound $W(q^2-1)<2q$. This completes the proof.
\end{proof}

\subsection{Applications of Theorem \ref{main}}
The following is straightforward.
\begin{proposition}\label{pracdiv}
Let $n$ be an $\F_q$-practical number. If $k$ is a positive integer such that
\begin{equation*}k\le n\cdot \left(\frac{1}{2}-\frac{1.06}{\log\log (q^n-1)}-\frac{\log 2}{\log q} \right),\end{equation*}
there exist primitive $k$-normals in $\F_{q^n}$.  
\end{proposition}

For $h(n, q)=\frac{1}{2}-\frac{1.06}{\log\log (q^n-1)}-\frac{\log 2}{\log q}$, we have $\lim\limits_{q\to \infty}h(n, q)=1/2$ uniformly on $n$. In particular, given $\varepsilon>0$, for $q$ sufficiently large, we can guarantee the existence of $k$-normals in the interval $[1, (\frac{1}{2}-\varepsilon)n]$ in the case when $n$ is $\F_q$-practical. In general, the bounds on character sums over $\F_{q^n}$ yield the factor $q^{n/2}$. In particular, this result is somehow sharp based on our character sums estimates. However, $h(n,q)$ goes to $1/2$ slowly and we do not have much control on $n$, since we are assuming that $n$ is $\F_q$-practical. 
Far from the extreme $n/2$, we have effective results:
\begin{corollary}
Let $q\ge 420$ be a power of a prime and let $n\ge 420$ be an $\F_q$-practical number. For any $k\in [1, n/4]$, there exist primitive $k$-normals in $\F_{q^n}$.
\end{corollary}
\begin{proof}
Note that $h(n, q)$ is an increasing function on $q$ and $n$. Also, $h(420, 420)>1/4$ and the result follows.
\end{proof}
The previous corollary gives a wide class of extensions having primitive $k$-normals in the interval $[1, n/4]$; for instance, we can consider $n=r^t, t\ge 1$, where $r\ge 420$ is a prime dividing $q-1$. In the special case when $n$ is a power of two, we obtain the following.
\begin{corollary}
Set $n=2^{t}, t\ge 2$ and let $q$ be a power of a prime. Additionally, suppose that $t\ge 9$ and $q\ge 259$ if $q$ is odd. Then, for $k\in [1, n/4]$, there exist primitive $k$-normal elements in $\F_{q^n}$.
\end{corollary}

\begin{proof}
Since $2$ always divides $p(q-1)$, from Proposition~\ref{practicals}, $2^t$ is $\F_q$-practical. For $q$ even, the result follows from Corollary~\ref{corpractical}. Suppose that $q$ is odd. According to Proposition~\ref{bound1}, $W(q^{2^t}-1)< 2q^{\frac{2^t}{t-1}}$ and we have the trivial bound $W(x^{2^t}-1)\le 2^{2^t}$. We can verify that, for $t\ge 9$ and $q\ge 259$, $2^{2^t+1}q^{\frac{2^t}{t-1}}\le q^{2^t/4}$ and the result follows from Theorem \ref{main}. 
\end{proof}
\section{Conclusions and an additional remark}
In this paper we have discussed the existence of primitive $k$-normal elements over finite fields. We recall some recent results on $1$-normal elements and partially extend them to more general $k$-normals. In particular, we obtain a sieve inequality for the existence of primitive $k$-normal elements in extension fields that contains $k$-normals. As an application, we give some families of pairs $(q, n)$ for which we can guarantee the existence of $k$-normals in $\F_{q^n}$ for $k\in [1, n/4]$: this corresponds to the first quarter of the possible values of $k$. For $q$ large enough, we can extend the range of $k$ to $[1, (1/2-\varepsilon)n]$, where $\varepsilon$ is close to $0$: in other words, we can asymptotically reach the first half of the interval $[1, n]$. As we have pointed out, the typical bounds for our character sums always have the component $q^{n/2}$ and, for this reason, in order to obtain results on primitive $k$-normals for $k\in [n/2, n]$, we need methods beyond the approach of Lenstra and Schoof or even better character sums estimates.

Here we make a brief discussion on the existence of primitive $k$-normals for $k$ in the other extreme, i.e., $k$ close to $n$. We recall that $0\in \F_{q^n}$ is the only $n$-normal element. Also, as pointed out in \cite{HMPT}, there do no exist primitive $(n-1)$-normals in $\F_{q^n}$: in particular, a primitive element and its conjugates cannot all lie in a ``line''. 

For the special case $k=n-2$ and $n>2$, we consider the following: set $q=p$ a prime, $n=p$ and let $a$ be a primitive element of $\F_p$. It can be verified that $f(x)=x^p-x-a$ is irreducible over $\F_p$ and any root $\alpha\in \F_{p^p}$ of $f(x)$ satisfies $(x-1)^2\circ \alpha=0$, i.e., $\alpha$ is $(n-2)$-normal. It is conjectured that such an $\alpha$ is always a primitive element of $\F_{p^p}$. This conjecture is verified for every $p\le 100$ and some few primes $p>100$. If this conjecture is true, we obtain an interesting example of a primitive element $\alpha$ of ``low normalcy'' (in the sense that $\alpha$ and their conjugates generates an $\F_q$-vector space of low dimension).

\begin{center}{\bf Acknowledgments}\end{center}
This work was conducted during a scholarship supported by the Program CAPES-PDSE (process - 88881.134747/2016-01) at Carleton University.


\end{document}